\documentclass[a4paper, 12pt]{amsart}
\usepackage{amsmath,amsthm,amsfonts,amssymb}
\theoremstyle{plain}

\newtheorem{thm}{Theorem}
\newtheorem*{thmA}{Theorem A}

\newtheorem{lemma}{Lemma}
\newtheorem{cor}{Corollary}
\newtheorem{step}{Step}
\newtheorem{claim}{Claim}
\newtheorem{remark}{Remark}
\newtheorem*{conj}{Conjecture}
\newtheorem*{thm*}{Theorem}

\newcommand{\Syl}{\operatorname{Syl}}

\newcommand{\Aut}{\operatorname{Aut}}
\newcommand{\Ker}{\operatorname{Ker}}
\newcommand{\Lin}{\operatorname{Lin}}

\newcommand{\Irr}{\operatorname{Irr}}

\newcommand{\F}{\mathcal{F}}
\newcommand{\E}{\mathcal{E}}

\title{On Distinct  Character Degrees}
\author{ Maria Loukaki }
\address{Dept. of Applied  Mathematics, University of Crete, Knossos Av. GR-71409, Heraklion-Crete, Greece}
\email{loukaki@gmail.com}

\begin{document}

\begin{abstract}
 Berkovich, Chillag and Herzog characterized all finite groups $G$ in which all the nonlinear 
irreducible characters  of $G$  have distinct degrees.
In this paper we extend this result showing  that a similar characterization holds for all finite solvable groups $G$  that contain a normal subgroup $N$, such that 
 all the irreducible characters of $G$ that do not contain $N$ in their kernel have distinct degrees.
\end{abstract} 

 \maketitle

\section{Introduction}
Let $N \ne 1$ be a normal subgroup of the finite group $G$.
We write $\Irr(G|N)$ for the set of irreducible characters of $G$ that do not 
contain $N$ in their kernel. 
We also say that {\it $ (G, N) $ has property (D) }, or that $G$ satisfies (D) with respect to $N$, 
 if all the  irreducible characters of $\Irr(G|N)$ have distinct degrees.
If $N= G'$ then  $(G, G')$ has property (D) exactly when all the nonlinear irreducible characters of $G$  have distinct degrees.  
These groups have been  fully characterized by Berkovich, Chillag and Herzog in \cite{bch}.  In particular, they have proved that a nonabelian group $G$ with the property that all its nonlinear irreducible characters have distinct degrees is either an extra special $2$-group, 
or a doubly transitive Frobenius group with a cyclic complement or a doubly transitive Frobenius group of order 72 having a quaternion complement. 

In the present note, we extend their result proving that  a similar characterization 
 holds  for a solvable  group $G$ that satisfies property (D)
with respect to a minimal normal subgroup $N$.
In particular we show (for the definition of Camina pairs see Section  2):
\begin{thmA}
Assume that $G$ is a solvable group, while $N$ is a minimal normal subgroup of $G$ of
 order $p^n$ for some prime $p$. 
If $(G, N)$ has property $(D)$ then  $(G, N)$ is a Camina pair, with $N$ 
being the unique minimal normal subgroup of $G$. Furthermore, $O_{p'}(G)= 1$ and the action of any $p'$-Hall subgroup $H$  of $G$ on $N$ is Frobenius.
 In particular $(G, N)$ has property (D) with $N < G$ if and only if  one of the following
 occurs \\
(1)  $G$ is a $2$-group of order $2^{2m+1}$ for some integer $m \geq 0$, while  $N = Z(G)$ is
  of order $2$.
 In addition, $G$ affords a unique faithful irreducible character whose  degree is $2^m$.\\
(2) $G$ is a Frobenius group with Frobenius kernel  $N$ and  complement of order $p^n-1$
which  acts   transitively on $N^{\#}$. In this case, $G$ affords a unique faithful 
irreducible character of degree $p^n-1$. \\
(3) $G$ is neither nilpotent nor Frobenius, but satisfies $O_{p'}(G)= 1= Z(G)$. In addition, 
if  $J = O^{p'}(G)$ then $(J, N)$ is also a Camina pair.
\end{thmA}

 Observe that if $G$ satisfies property (D)
with respect to $M \unlhd G$, then $G$ satisfies (D) with respect to $N$ for every normal subgroup $N$ 
of $G$ with $N \leq M$ (since  $\Irr((G|N ) \leq \Irr(G|M)$  when $N \leq M$).
 Hence Theorem A characterize  the solvable groups that satisfy property (D) with
respect to  some normal subgroup $M$.

The groups of type (3) in Theorem A,  have been already classified by Kuisch in \cite{ku}, 
where he proves,  see Theorem B in \cite{ku},  the following:
\begin{thm}[Kuisch]
Let $(G, N)$  be  a Camina pair, $G$ a solvable group and $N$ a $p$-group. Then 
$O_{p'}(G)=1$. If $J = O^{p'}(G)$ then $(J, N)$ is also a Camina pair and  
 one of the following holds:\\
(i) $J \in \Syl_p(G)$ or \\
(ii) $O_p(J)= O_p(G), \,  O_{p, p', p}(J)= J, \,  O_{p, p'}(J)/O_p(J)$ is cyclic of odd 
 order, $J/O_{p, p'}(J)$ is an abelian $p$-group, and  $J/O_{p, p'}(J)$ acts fixed point-freely 
on $O_{p, p'}(J)/O_p(J)$ or \\
(iii) $p=3, \,  O_3(J)= O_3(G), \,   O_{3, 3', 3}(J)= J, \, 
 O_{3, 3'}(J)/O_3(J)$ is  the direct product of a quaternion group of order $8$ and  a  cyclic
 of odd order, $J/O_{3, 3'}(J)$ is abelian, and 
 $$
[J/O_{3, 3'}(J), O_{3, 3'}(J)/O_3(J)]= O_{3, 3'}(J)/O_3(J).
$$
\end{thm}

Hence if $G$ is a solvable group that satisfies (D) with respect to a 
minimal normal subgroup $N$, 
then $G$ is either a $2$-group or a Frobenius group or one 
of the 3 types of groups that appear in Kuisch's list.
Even though we have a clear image of the groups of type (i) in this 
list that in addition have property (D), see  Corollary \ref{cor1} of Section 2 below, 
we are quite uncertain how the other two types can coexist with property (D), as we have not
 been able to construct such examples.

Theorem A was inspired by the paper \cite{bik} of Berkovich, Isaacs and Kazarin, 
 (see Corollary 4.5 in \cite{bik}) where under the same 
hypothesis of Theorem A several properties of the 
 solvable group $G$ are given.
The key tool in proving Corollary 4.5 in \cite{bik}, along with many other interesting 
theorems,  was the following result (Theorem E in \cite{bik}):\\
Let  $N \unlhd G$ be a $p$-group  while $G/N$ is solvable, and let  $\theta \in \Irr(N)$ be $G$-invariant. If the members of $\Irr(G|\theta)$ have distinct degrees then $|\Irr(G|\theta)|=1$, and $G$ is a $p$-group.\\
%Let $G$ be a finite group and $\rm{Proj}(G, \alpha )$ be the set of irreducible %complex projective characters of $G$ with cocycle $\alpha$. 
%In  \cite{higgs} Higgs formed the following  conjecture:
% $\rm{Proj}(G, \alpha )$ consists of a single element or contains at least two
%elements of the same degree.
It turns out that this is a case of the following  conjecture formed by Higgs in \cite{higgs}.
\begin{conj}
Let $N \unlhd G$ and $\theta \in \Irr(N)$ be $G$-invariant.
If all the members of $\Irr(G |\theta)$ have  distinct degrees then there is only one irreducible character of $G$ lying above $\theta$.
\end{conj}
In the same paper  (\cite{higgs}) Higgs proved his conjecture when $G$ is supersolvable or has odd order.  His methods use projective characters 
(actually the conjecture itself was also formed in terms of projective characters). 
In the last section of this note we give a different proof of Higgs's
results. In particular we prove 
\begin{thm}\label{thm1}
Let $G$ be a finite group, $N$ a normal subgroup of $G$ and $\theta$ an 
irreducible $G$-invariant character of $N$. Assume further that  all irreducible 
characters of $G$ lying above $\theta$ have distinct degrees.
If $G/N$ is a supersolvable group then $\theta$ is fully ramified in $G/N$.
\end{thm}

\begin{thm}\label{thm2}
Let $G$ be a finite group, $N$ a normal subgroup of $G$ and $\theta$ an 
irreducible $G$-invariant character of $N$. Assume further that  all irreducible 
characters of $G$ lying above $\theta$ have distinct degrees.
If $G/N$ has odd order then $\theta$ is fully ramified in $G/N$.
\end{thm}

{\bf Acknowledgments}
I would like to thank E. C. Dade, 
for  his valuable remarks  that improved the original version of this note.
I would also like to thank M. Isaacs, for pointing out Corollary C and 
   Lemma 6.1 in \cite{is1}, and thus simplifying  my proofs.

\section{Pairs $(G, N)$ that satisfy property (D)}

Let $N \ne 1$ be a normal subgroup of the finite group $G$.
We say the $(G, N)$ is a Camina pair if it satisfies the following hypothesis

(F2)  If $x \in G \setminus N $, $x$ is  conjugate to $xy$ for all $y \in N$.

\noindent
Camina pairs were first introduced by Camina in \cite{cam}. For their definition Camina used in
 \cite{cam} a character-theoretic approach (see  hypothesis  (F1) in \cite{cam}), and  
proved that his  definition  is equivalent to hypothesis (F2) above.
In \cite{chima}, Proposition 3.1, Chillag  and MacDonald showed that
$(G, N)$ is a Camina pair if and only if 
for every $x \in G \setminus N$ we have $|C_G(x) | =
 |C_{G/N}(xN)|$. 
 (In their terminology a Camina pair $(G, N)$ is said to be a pair that has (F2).)

\begin{lemma}\label{l2.0}
Let $N \unlhd G$, where $N$ is a $p$-group and $G/N$ is solvable. 
Let $\theta \in \Irr(N)$ be $G$-invariant  and assume that the members of $\Irr(G| \theta)$ 
have distinct degrees.
Then $\theta$ is fully ramified in $G/N$  and $G$ 
is a $p$-group. 
\end{lemma}

\begin{proof}
Theorem E in \cite{bik}
\end{proof}

The following is part of Corollary 4.5 in \cite{bik}.

\begin{lemma}\label{l2.2}
Assume that  $(G, N)$  has property (D),    with $N$ a minimal
 normal $p$-subgroup of a nonabelian 
 solvable group $G$. Then every non principal character $\lambda$  of $N$ is fully ramified
 in  $G(\lambda)/N$. So for every such $\lambda$, the group $G(\lambda)$ is a $p$-group.
Hence the action of any  $p'$-subgroup  $Q$ of $G$ on $N$ is Frobenius.
In addition, if $\lambda, \mu $ are non principal linear characters of $N$ whose stabilizers
 have equal order in $G$
then they are $G$-conjugate.
\end{lemma}

\begin{remark}\label{r2.1}
Observe that since every non principal character $\lambda$ of $N$  is fully ramified with
 respect to $G(\lambda)/N$, 
  there is only one irreducible character of $G$ lying above the $G$-conjugacy class  of
  $\lambda$.
\end{remark}

\begin{lemma}\label{l2.1}
 Assume that  $(G, N)$  has property (D),    with $N$ a minimal normal $p$-subgroup 
of a nonabelian 
 solvable group $G$. If  $N \leq Z(G)$ then  $N = Z(G)$ has order $2$, while  $G$ is a
 $2$-group
 of order $2^{2m+1}$ for some integer $m \geq 1$. Furthermore, $N$ is the unique minimal
 normal subgroup of $G$, 
 while $G$ affords a unique faithful irreducible character. Its degree is $2^m$. 
\end{lemma}

\begin{proof}
If  $N \leq Z(G)$ then every irreducible character of $N$ is $G$-invariant. Let $\lambda \in \Lin(N)$
 be a non principal linear character of $N$.
 If $\chi \in \Irr(G)$ is the unique irreducible character  that lies above 
$\lambda$, then $\chi(1) = |G:N|^{1/2}$. 
If $N $ is properly contained in $Z(G)$, then 
$\lambda$ extends to $Z(G)$. Let $\lambda' \in \Lin(Z(G))$ be such an extension. 
Then $\lambda'$ is also fully ramified in $G/Z(G)$, as $\chi$ is the 
unique character of $G$ above it,  and thus $\chi(1) = |G:Z(G)|^{1/2}$.
We conclude that $N = Z(G)$. 
Furthermore, because $(G, N)$ has property (D), we can only have one irreducible character 
of $G$ of degree $|G:N|^{1/2}$. Thus the order of $N$ is $2$. In addition $|G:N|= 2^{2m}$ for some $m \geq 1$.
Note also  that the irreducible character of $G$ that lies above the non principal character of $N$
is the only faithful character of $G$.
\end{proof}

We can now prove Theorem A.

\begin{proof}[Proof of Theorem A]
Assume that the pair   $(G, N)$ satisfies  property $(D)$, where  $G$ is a nonabelian group. So $G > N$.

Because every non principal linear  character $\lambda$ of $N$ is fully ramified with 
respect to $G(\lambda)/N$, every irreducible character of $G$ that lies above such a $\lambda$
 is not linear. Hence  every linear character of $G$ restricts trivially on $N$, and thus 
$N \leq G'$.  Note also that Lemma \ref{l2.2} implies that the action of any $p'$-subgroup of $G$ on  $\Irr(N)$ and thus on  $N$ is Frobenius.

\setcounter{claim}{0}
\begin{step}
$N$ is the unique minimal normal subgroup of $G$.
\end{step}

\begin{proof}
Assume not. Let $M$ be another minimal normal subgroup of $G$. 
Then $N $ centralizes $M$. 
Let $\lambda$ be a non trivial linear character of $N$. 
Then $M \times N$ is a subgroup of $G(\lambda)$.  The linear character 
 $\alpha= 1_M \times \lambda$ of $M \times N$ has the same stabilizer in $G$ as $\lambda$. 
Furthermore, there is only one irreducible character $\theta$ of 
 $G(\lambda) = G(\alpha)$ lying above 
$\lambda$, and thus above $\alpha$, since $\lambda$ is fully ramified. Thus $\alpha$ is also 
fully ramified with respect to $G(\lambda)/ (M\times N)$. This forces the degree of $\theta$ to
 equal $|G:NM | ^{1/2} = |G:N|^{1/2}$. Hence $M = 1$, and  the claim is proved.
\end{proof}

Because $N$ is the unique minimal normal subgroup of $G$, it is contained in the kernel of every
 non faithful irreducible character of $G$. Thus the set  
 $\Irr(G| N)$  consists of the  faithful irreducible characters of $G$, while every 
other irreducible character 
 of $G$ has $N$ in it kernel.

\begin{step}
$(G, N)$ is a Camina pair.
\end{step}

\begin{proof} 
Let $\chi \in \Irr(G |N)$ and assume that $\lambda \in \Lin(N)$ lies under $\chi$. 
If $\theta$ is any non-faithful character of $G$ then $\theta \chi$ 
is a character of $G$ whose restriction to $N$ equals $\theta(1) \cdot \chi_N$. 
Hence $\theta \chi$ lies above the $G$-conjugacy class of $\lambda$. 
In view of Lemma \ref{l2.2} only $\chi$ lies above the $G$-conjugacy class of $\lambda$.
So $\theta \chi = \theta(1) \chi$. 
Hence $\chi(g)= 0$  for all $g \in G \setminus \Ker(\theta)$. 
Since this is true for all $\theta$ with $N \leq \Ker(\theta)$, 
we have that $\chi $ vanishes off
$$
\cap\{ \Ker(\theta) | \theta \in \Irr(G), \, N \leq \Ker(\theta) \} = N 
$$
Hence every character of $ \Irr(G| N)$ vanishes off $N$.
Now assume that  $x \in  G \setminus N$. Then 
$$
|C_G(x)|  = \sum_{\{ \theta \in \Irr(G) | N \leq \Ker(\theta) \}} (|\theta(x)|^2) + 0 = |C_{G/N}(xN)|.  
$$
This proves the claim. 
\end{proof}

\begin{step}
If $G$ is a nilpotent group then it is of type (1).
\end{step}

\begin{proof}
Let $G$ be a nilpotent group. Then it is  a $p$-group, since 
the action of any $p'$-subgroup $Q$ of $G$  on $N$ is Frobenius.
 Because  $N$ is  minimal normal we have  $N \leq Z(G)$. Now Lemma \ref{l2.1} implies that $G$ 
is of type (1), and the claim follows.
\end{proof}

Because $G$ affords a faithful irreducible character its center $Z(G)$ is either cyclic
 or trivial. 
If $Z(G)$ is nontrivial then $N \leq Z(G)$ since $N$ is unique minimal 
subgroup of $G$. So   Lemma \ref{l2.1} implies that $G$ is of type (1). 
Hence $Z(G)= 1$. 

Assume  now that $G$ is  a Frobenius group with kernel $N$. 
So every non principal character of $N$ induces 
irreducibly to $G$. Since $(G, N)$ has property (D) we must 
only have one element in $\Irr(G |N)$. So if  $\chi \in \Irr(G | N)$ then $\chi$ lies above the 
$|N|- 1$ non principal linear characters of $N$. Since the degree of $\chi = \lambda^G$ equals  
$|G:N|$, we conclude that $|G:N| = |N|-1$. So $G$ is of type (2). 

Finally, if  $G$ is neither  a Frobenius group with kernel $N$ nor a nilpotent group, 
then Kuisch's theorem  (Theorem 1), implies that $G$ is of type (3).
This  completes the proof of the theorem.
\end{proof}

\begin{cor}\label{cor1}
If $G$ is of type (3) with $J = O^{p'}(G)$ being a $p$-Sylow subgroup of $G$, 
 then   $G = HP$ where $H$ is a $p'$-Hall  subgroup of $G$ of order $p^n -1$, while the action of $H$ 
on $N$ is Frobenius and it is   transitively on $N^{\#}$.
 In addition, $\Irr(G|N)$ consists of a unique irreducible character whose degree
 equals $(p^n-1) (|P|/ p^n)^{1/2}$.
\end{cor}

\begin{proof}
Let $G$ be of type (3) and let  $J$ be a $p$-Sylow subgroup of $G$. 
(In view of the notation in  \cite{chima}, $(G, N)$ is said to have F2(p) with $N$ a $p$-group.) 
Let $J = P$ be the normal  $p$-Sylow subgroup of $G$, then  Proposition 3.4 in \cite{chima}, 
implies 
$Z(P) \leq N$. So $Z(P) = N$, because $N$ is minimal.
According to Lemma 4.3 in \cite{chima}, if $H$ is a $p'$-Hall subgroup of $G$, then 
  $HN$ is a Frobenius group.

Because $N = Z(P)$ every irreducible character of $N$ is $P$-invariant. Furthermore, as $P$ is a
$p$-Sylow subgroup of $G$, Lemma \ref{l2.2} implies that 
$G(\lambda) = P$ for every non principal linear character $\lambda \in \Lin(N)$.
Hence all the non principal linear characters of $N$ form a single $G$-orbit. In view of Remark \ref{r2.1}, 
we conclude that we only have one character $\chi  \in \Irr(G |N)$. 
Note that 
$$
\chi(1) = |G : P| |P:N|^{1/2}= |H| |P:N|^{1/2}. 
$$
In addition, because $G = H P$ and $P$ centralizes $N$, the single  $G$-orbit of $\Lin^{\#}(N)$ is a 
single $H$-orbit.  Since $HN$ is a Frobenius group we conclude that we only have one irreducible character
  $\theta \in  \Irr(HN |N)$. Furthermore, $\theta(1) = |H|= p^n -1$. 
\end{proof}

\section{Irreducible characters of distinct degrees over an invariant character}
We start with the following known result (see Lemma 12.5 in \cite{wolf})
\begin{lemma}\label{l2}
Let $(G, N, \theta)$ be a character triple.  Assume further that $\theta$ is fully 
ramified in $G/N$ while $G/N$ is an abelian group. 
Then $G/N= C_1 \times C_2$ where $C_1 \cong C_2$.
\end{lemma}

We first  prove the supersolvable case. 

\begin{proof} [Proof of Theorem \ref{thm1}]
Work using induction on $|G|$ and then on  $|G:N|$.
Let $s$  be the number of irreducible characters of $G$ lying above $\theta$.
(So $s  > 1$.)

Repeated applications of Clifford's theory along with the inductive hypothesis implies that 
  $N$ is a cyclic central subgroup of $G$ while  $\theta$ is
 a linear faithful character.

The main step of the proof is the following
\begin{claim} \label{cl1}  $s=2$, and  if $\{ \chi_1, \chi_2 \} = \Irr(G |\theta)$ 
then $\chi_1(1) = (|G:N|/p)^{1/2}$ and $\chi_2(1) = (p-1)^{1/2}  \cdot (|G:N|/p)^{1/2}$, 
where $p$ is a prime divisor of a chief section $M/N$ of $G$.
\end{claim}

\begin{proof}
Let $M/N$ be a chief section of $G$ with $|M:N| = p$.
Let $\phi_1, \phi_2, \dots, \phi_p$ be all the distinct extensions of $\theta$ 
to $M$. If $G(\phi_i)$ is the stabilizer of such an extension to $G$, then Clifford's 
theorem implies that induction defines a bijection between 
$\Irr(G| \phi_i)$ and $\Irr(G(\phi_i) | \phi_i)$. Because every character in
 $\Irr(G |\phi_i)$  lies above $\theta$, we conclude that all
 the irreducible characters
of $G(\phi_i)$ lying above $\phi_i$ have distinct degrees. Hence the inductive
 hypothesis implies that there exists only one character of $G(\phi_i)$ lying
 above $\phi_i$. We conclude that there exists only one irreducible character of $G$
above every $G$-orbit of $\phi_i$. 
Thus $s=|\Irr(G| \theta)|$ equals the number of orbits of the action of $G$ on 
the set or irreducible characters $\phi_1, \dots, \phi_p$. 
Let $\phi_1, \phi_2, \dots, \phi_s$ be a representative for each one of these 
orbits. Of course if the index of $G(\phi_i)$, for some $i=1, \dots, s$, 
  in $G$ is $p$ then we only have one
 orbit and the theorem holds. So we assume that $|G :G(\phi_i )| \ne p$ for all
 $i=1, \dots, s$. Let $\psi_i$ for $i=1, \dots, s$, be the unique irreducible
 character of $G(\phi_i)$ that lies  above $\phi_i$ and induces irreducibly to $G$.
Furthermore, let $\chi_i = \psi_i^G$, for all $i=1, \dots, s$. So
 $\{ \chi_i \}_{i=1}^s$ 
are all the irreducible characters of $G$ lying above $\theta$.

Note that the inductive hypothesis also implies that 
$\psi_i(1) = |G(\phi_i):M|^{1/2} \cdot \phi_i(1)$, since $\phi_i$ is fully ramified
 in $G(\phi_i)/ M$, for all $i=1, \dots, s$. 
Because  $\phi_i$ induces $\chi_i$ to $G$ we get   
$$
\chi_i(1)=|G:G(\phi_i)| \cdot  |G(\phi_i):M|^{1/2} \cdot \phi_i(1), 
$$
for all $i=1, \dots, s$. 
But $\chi_i(1)$ are all distinct, so we conclude that $|G:G(\phi_i)|$ 
are also distinct for all $i=1, \dots, s$.

Let $H/N$ be a $p'$-Hall subgroup of $G/N$. Then $H/N$ acts on $\Irr(M/N)$ 
as it acts on $M/N$. It also acts on the set $\Irr(M| \theta)$. In addition, the 
group $\Irr(M/ N)$ of linear
 characters acts transitively on $\Irr(M |\theta)$ by multiplication.
Hence Glauberman's lemma implies that there exists an irreducible  character in the set 
$\Irr(M | \theta)$ that is $H/N$-invariant. Let $\phi_1$ be such.
So  the index of $G(\phi_1)$ in $G$ is a power of $p$. Thus is it either $1$ or $p$. 
It can't be $p$ or else all the $\phi_i$, for $i=1, \dots, p$, form a single $G$-orbit. 
Hence $\phi_1$ is $G$-invariant. As we have a unique irreducible character $\chi_1$ in 
$G= G(\phi_1)$ above $\phi_1$, we conclude that 
$$
\chi_1(1) = |G:M|^{1/2} =(|G:N|/p)^{1/2} 
$$ 

According to Gallagher's theorem  (Corollary 6.17 in \cite{is}), distinct linear characters
$\lambda \in \Irr(M/N)$ provide distinct irreducible characters $\lambda \cdot \phi_1$ 
of $M$ lying above $\theta$. So $G(\phi_1) \cap G(\lambda \phi_1) \leq G(\lambda)$, for all 
such $\lambda$. On the other hand $G(\lambda) \leq G(\phi_1) \cap G(\lambda \phi_1)$, 
because $G(\phi_1) = G$. 
We conclude that $\cap_{i=1}^{p}G(\phi_i) = \cap_{i=1}^{p}G(\lambda_i)$.
The later group equals $C_G(M/N)$  and thus 
\begin{equation}\label{e.1}
\cap_{i=1}^{p}G(\phi_i) = C_G(M/N)=C.
\end{equation}

The group $G/C$ acts faithfully and irreducibly  on the $p$-group $M/N$.
So $G/C$ is a subgroup of $\Aut(M/N)$, and thus $G/C$ is a cyclic group whose 
order is a divisor of $p-1$.
In addition, $G/C$  acts faithfully and irreducibly  on $V=\Irr(M/N)$. 
Because $G/C$ is cyclic  every non trivial orbit in $V=\Irr(M/N)$ has size $|G/C|$.
It is easy to see that there is a one--to--one correspondence between the 
 $G$-orbits of $V$ and the $G$-orbits of all the  extensions $\phi_1,
\lambda_1 \phi_1, \cdots ,\lambda_{p-1} \phi_1$ of $\theta$ to $M$, 
where $\lambda_0 = 1$ corresponds to $\phi_1$ and 
the size of two corresponding orbits is the same.
Hence all the  $G$-orbits of $\phi_i$, for $i=2, \dots, p$,  have size $|G/C|$.
As these orbit sizes are distinct we conclude that $\phi_2, \dots, \phi_{p-1}$
form a single $G$-orbit, while $C = G(\phi_i)$, for $i=2, \dots, p$ and $|G:C|= p-1$.

So $s=2$ and we have only one irreducible 
character $\chi_2 \in \Irr(G)$ lying above $\phi_2, \dots, \phi_p$, 
and induced from $\psi_2 \in \Irr(C)$. As $\phi_2$ is fully ramified in $C/M$  
we conclude that 
$$
\chi_2(1) = (p-1) \cdot |C:M|^{1/2}=(p-1)^{1/2} \cdot (|G:N|/p)^{1/2} 
$$
Hence Claim \ref{cl1} follows.
\end{proof}

According to Corollary C in \cite{is1}, if we only have two irreducible characters $\chi_1, \chi_2$ 
of $G$ lying above an irreducible character of a normal subgroup $N$ of $G$ 
then $\chi_1(1)= \chi_2(1)$. This contradicts Claim \ref{cl1}, and thus the theorem follows.
\end{proof}

The following is Lemma 2.1 in \cite{esp}
\begin{lemma} \label{esp}
Let $G$ be a finite group of odd order having a faithful and irreducible quasiprimitive module $V$ over a finite field $\F$ of odd characteristic. If $F(G)$ in noncyclic , then $V$ contains 
at least two regular $G$-orbits.  
\end{lemma}

Using the above lemma we can prove 

\begin{lemma}\label{orb}
Assume that $G$ is a finite group of odd order, that acts irreducibly  on a
nontrivial $\F$-vector space $V$, where $\F$ is a field of odd  characteristic.
  Assume further that the orbits of $G$ on $V$ 
have distinct sizes. Then  $G$  acts transitively on $V - \{ 0 \}= V^{*}$.
\end{lemma}

\begin{proof}
Work using induction on $|G|$ and on $\dim_{\F}V$.
Let $G, V$ be a smallest counterexample.
We can assume that $G$ acts faithfully on $V$.

Let $\E$ be a splitting field for $G$ with $\F \leq \E$.
Write $V^{\E}= V \times_{\F} \E$. Then  $V^{\E}= W_1 \oplus \cdots \oplus W_k$, 
where $W_i$ are non--isomorphic irreducible $\E G$-modules, conjugate under the Galois group 
of the field extension $\E :\F$.
Furthermore, the orbits of $G$ on each $W_i$ 
have distinct sizes since these are some of the orbits of $G$ on $V$. 
Hence the inductive hypothesis implies that $G$ 
acts transitively on $W_i^{*}$, for all $i=1, \dots, k$. 
If $k\geq 2$ then we would get exactly  $k$ orbits of $G$ on $V$ of size $W_1^{*}$, which 
contradicts the hypothesis of the lemma.we are done by similar arguments to those used for 
the supersolvable case, 
 Thus we may assume that $k=1$ and thus $V$ is an absolutely irreducible 
$\F(G)$-module.

Let $N$ be a normal subgroup of $G$. Then Clifford's theorem implies that 
$V_N= U_1\oplus U_2 \oplus \cdots \oplus U_n$, where $U_i$ are the homogeneous
 components of $V_N$. Let $I$ be the inertia subgroup  of $U_1$ in $G$. Then $U_1$ 
is an  irreducible $\F I$-module, while $U_i= U_1 \cdot g_i$, for some $g_i \in G-I$ 
for all $i=1, \dots, n$,   and $n = |G:I|$. If $O$ is an orbit of the action of $I$ on $U_1$, 
then the union $\cap_{i=1}^n (O \cdot g_i)$ is actually a full orbit of the action of $G$ on 
$V$ say $\Omega$. So $|\Omega| = n \cdot |O|$. This way we get a one-to-one correspondence between 
the orbits of $G$ on $V$ and those of $I$ on $U_1$. 
 We conclude that the orbits of the action of $I$ on $U_1$ have distinct sizes.
Thus the $\F I$ irreducible module $U_1$ satisfies all the hypothesis of the lemma. 
So $I$  acts transitively on $U_1^{*}$, which implies that   $G$ acts transitively on $V^{*}$.
This contradicts the inductive hypothesis. Hence we get that $V$ is a
 primitive $\F G$-module. 

So $V$ is a primitive irreducible faithful $\F G$-module. If $F(G)$ is not cyclic then 
Lemma \ref{esp} implies that $V$ contains at least two regular $G$-orbits. But the $G$-orbits 
have all distinct sizes, and thus $F(G)$ is a cyclic group. 
Because $V$ is absolutely irreducible faithful primitive $\F G$-module, 
all the abelian normal subgroups 
of $G$ are cyclic and central.  We conclude that $F(G)$ is a cyclic central subgroup of $G$. 
But $G$ is solvable hence $C_G(F(G) \leq F(G)$. So $G= F(G)$ is a cyclic group. 

Now $V$ is an absolutely irreducible faithful $\F G$-module where $G$ is a cyclic group. 
Then $\dim_{\F} V= 1$, and $\F$ contains a primitive $t$-th root of $1$, where $t = |G|$.
 Furthermore, there exists a primitive $t$-th root of $1$,  $\zeta \in \F$ so that 
$v^g = \zeta v$, where $g$ is a generator of $G$ and  $v$ is any element of $V$.
Then the elements $ \{ v, \zeta v , \zeta^2 v, \cdots, \zeta^{t-1} v \} $ 
are all distinct and form 
 a $G$-orbit of $V$ of size $t$. 
This is true for any $v \in V^{*}$,  but we can't have more than 
one orbit of size $t$. Hence we only have one such orbit and thus 
$G$ acts transitively on $V^{*}$. 
Note also that  $|V| = |G| +1$.
\end{proof}

We can now prove Theorem \ref{thm2} that we restate 
\begin{thm*}
Let $G$ be a finite group, $N$ a normal subgroup of $G$ and $\theta$ an 
irreducible $G$-invariant character of $N$. Assume further that  all irreducible 
characters of $G$ lying above $\theta$ have distinct degrees.
If $G/N$ has odd order then $\theta$ is fully ramified in $G/N$, that is 
there exists only one irreducible character of $G$ lying above $\theta$.
\end{thm*}

\begin{proof}
We use induction on $|G|$  and then on  $|G:N|$. 
Without loss we can assume that $N$ is a cyclic central subgroup of $G$ while $\theta$ is a $G$-invariant
 faithful character of $N$.

The character $\theta$ is either  fully ramified with respect to the
 chief section $M/N$ or it extends to $M$. 
In the first case we are done  by induction. If $\phi$ is the unique irreducible 
character of $M$ above $\theta$, then $\phi$ is $G$-invariant and all the characters 
of $G$ above it have distinct degrees.

So assume that  $\theta$ extends to $M$.
Let $|M/ N|= p^n$, where $p$ is an odd  prime number.
We also write $\phi_1, \dots, \phi_{p^n}$ for the extensions of $\theta$ to $M$.
Then all the irreducible characters of $G(\phi_i)$ lying above $\phi_i$ have distinct degrees,
for all $i=1, \dots, p^n$. 
Hence induction implies that the character $\phi_i$ is fully ramified in its 
stabilizer $G(\phi_i)$, for all $i=1, \dots, p^n$. Therefore there exits only one irreducible character of $G$ above every $G$-orbit of $\phi_i$. So if $s = |\Irr(G | \theta)|$, then $s$ is
also the number of orbits  of the $G$-action on $\{ \phi_i \}_{i=1}^{p^n}$.
Let $\phi_1, \dots, \phi_s$  be a representative of each one of these orbits.
Observe  that every $G$-orbit has size $|G:G(\phi_i)|$, where $\phi_i$ is its representative,  for  some $i=1, \dots, s$,
while the corresponding irreducible character $\chi_i$  of $G$ lying above  that $\phi_i$ has degree 
$$
\chi_i(1) = |G:G(\phi_i)| \cdot |G(\phi_i):M|^{1/2},
$$
for all $i=1, \dots, s$.
Hence  the $G$-orbits on $\{ \phi_i\}_{i=1}^{p^n}$ have distinct sizes.

{\bf Case 1:} Assume there exists an $i=1, \dots, s$ so that $\phi_i$ is $G$-invariant. 

Let $\phi_1 = \phi$ be a $G$-invariant extension of $\theta$ to $M$.
According to Gallagher's theorem  (Corollary 6.17 in \cite{is}), distinct linear characters
$\lambda \in \Irr(M/N)$ provide distinct irreducible characters $\lambda \cdot \phi$ 
of $M$ lying above $\theta$. So $G(\phi) \cap G(\lambda \phi) \leq G(\lambda)$, for all 
such $\lambda$. On the other hand $G(\lambda) \leq G(\phi) \cap G(\lambda \phi)$, 
because $G(\phi) = G$. 
We conclude that $\cap_{i=1}^{p^n}G(\phi_i) = \cap_{i=1}^{p^n}G(\lambda_i)$.
The later group equals $C_G(M/N)$ as $M/N$ is an elementary abelian group, and thus 
$$
\cap_{i=1}^{p^n}G(\phi_i) = C_G(M/N)=C.
$$  
The group $G/C$ acts faithfully and irreducibly  on $M/N$.
 So it acts faithfully and irreducibly  on $V= \Irr(M/N)$. 
We consider the latter as a finite vector space over $\F _p$. 
It is easy to see that there is a one--to--one correspondence between the 
 $G$-orbits of $V$ and the $G$-orbits of all the  extensions $\phi,
\lambda_1 \phi, \cdots ,\lambda_{p^n-1} \phi$ of $\theta$ to $M$,
where $\lambda_0 = 1$ corresponds to $\phi$ and
the size of two corresponding orbits is the same.
Hence all the $G$-orbits of $V$ have distinct sizes.
Then Lemma \ref{orb} implies that $G/C$ acts transitively on $V^*$. So we have only
 two orbits the trivial one and another one of length $p^n-1$. Thus we  only have 
two irreducible characters  $\chi_1, \chi_2$, lying above  $\theta$.
Hence   Corollary C in \cite{is1} implies  that $\chi_1(1) = \chi_2(1)$,
 contradicting the assumptions  of the theorem. Hence we can only have one 
irreducible character of $M$ lying above  $\theta$  in this case, and thus
  the theorem  follows.

{ \bf Case 2:} For for all $i=1, \dots, p^n$,  we have $G(\phi_i) < G$.

First note that $p$ divides $|N|$. If not, then Corollary 6.28 in \cite{is}
implies that $\theta$ has a unique canonical extension  $\phi \in \Irr(M)$ such that
$o(\phi)= o(\theta)$. Because $\theta$ is $G$-invariant the uniqueness of $\phi$ makes
 it also $G$-invariant. So  we are back to Case 1. Hence  $p/|N|$.

Let  $H/N$ be  a $p'$-Hall subgroup of  $G/N$. Then $H/N$ acts on $\Irr(M/N)$ as well as on
$\Irr( M | \theta )$.  The group $\Irr(M/N)$ acts transitively on $\Irr(M | \theta)$
by multiplication. Hence Glauberman's lemma implies that there exists $\phi_1$ above
$\theta$ with $|G: G(\phi_1)| = p^s$, for some $0 < s < n$.
Let $\chi_1$ be the unique irreducible character of $G$ lying above $\phi_1$.
Then
$$
\chi_1(1)= p^{-t/2} \cdot (|G:N|)^{1/2}, 
$$
where $t = n-s> 0$.
Now assume  we can find another chief section $K/N$ of $G$ with $|K:N| = q^{n'}$,
 for some prime $q \ne p$. Then $\chi_1$ lies above some extension character $\phi_1'
 \in \Irr(K | \theta)$. Note that $G(\phi_1') < G$ or else we would be in Case 1
for the chief section $K/ N$. So, as earlier, $\phi_1'$ is fully ramified in
$G(\phi_1') / K$, and the unique irreducible character $\psi_1' $ of $G(\phi_1')$
lying above $\phi_1'$ induces irreducibly to $\chi_1$.
So
\begin{multline*}
p^{-t/2} \cdot (|G:N)^{1/2} = \chi_1(1)=  |G:G(\phi_1') | \cdot (|G(\phi_1'):N|/q^{n'})^{1/2}=\\
(|G:N|)^{1/2} \cdot  (|G:G(\phi_1')|)^{1/2} \cdot q^{-n'/ 2},
\end{multline*}
which implies that
$q^{n'}= |G:G(\phi_1') | \cdot p^t$. This is impossible if $p \ne q$.
We conclude that $O_{p'}(G/ N)= 1$.
Thus, if $F$ is the Fitting subgroup of $G$ then $F= N_{p'} \times P$, where
$P = O_p(G) > N_p$.

\begin{claim} \label{ex-abel}
Every characteristic abelian subgroup of $P$ is cyclic.
\end{claim}
 Assume not. Then we can find an abelian characteristic subgroup $A$ of $F$
containing $N$ that is not cyclic. Hence there exists an abelian
non cyclic normal subgroup  $L$ of $G$ such that
 $L = N \times E \leq  N \cdot \Omega_p(P) \unlhd G$, where $E$ is an
elementary abelian  $p$-group
of order $p^k$ for some $k=1,2,  \dots $, and $L/N$ is a chief section of $G$.

Let $\psi_i$, for $i=1, \dots, p^k$,  be the extensions of $\theta$ to $L$.
Then $G(\psi_i) < G$, for all $i=1, \dots, p^k$,
  or else we would be done by Case 1. Furthermore, there exists an extension say
$\psi_1$ of $\theta$ to $L$, 
such that $|G: G( \psi_1) |$  is a power of $p$.
Clearly $|G:G(\psi_1)|  < p^k$, or else $\{ \psi_i\}_{i=1}^{p^k} $ form a unique  $G$-conjugacy class, which implies the theorem.

Let $S= C_G(L/N)$. Then  $S = C_G(L)$. Assume not, and let $X = C_G(L/N)/ C_G(L)$.
It is easy to see that $C_G(L/N) \cap G(\psi_1) = C_G(L)$. Hence
$|X| = |(G(\psi_1) \cdot C_G(L/N))/ G(\psi_1)| \leq |G:G(\psi_1)| < p^k$.
Hence $X$ is a $p$-group of order strictly less than $p^k$.
Because $E$ is an elementary abelian $p$-group of order $p^k$,  there exists
$e_1, \dots, e_k \in E$ of order $p$ so that $E = \prod_{i+1}^{k} <e_i>$.
Let $x \in X$. Then for every $i=1, \dots, n$ there exists $n_{i, x} \in N$ so that
$$
e_i^{x}= n_{i, x} e_i.
$$
Hence for every $i=1, \dots, k$, we   get a group homomorphism $f_i : X \to N$
so that
$f_i(x) = n_{i, x}$. Let $K_i$ be the kernel of $f_i$. Then $\cap_{i=1}^k K_i= 1$,
 because $X$ acts faithfully on $L = N E$  while $N $ is central in $G$.
Hence the map $f: X \to N^k$
defined as   $f(x) = (f_1(x), f_2(x),  \dots, f_k(x))$ is a group monomorphism.
Therefore $X$ is isomorphic to a subgroup $Y$ of $N^k$  of order  strictly less that $p^k$.
 On the other hand $|Y|$ is the product of the orders of the images  $|f_i(X)|$,
 for $i=1, \dots, k$. Since $|f_i(X)| = |X|/ |K_i|$ is a power of $p$, 
we conclude that there exists some $i$ so that $K_i = X$. 
Hence there exists some $i$ so that $X$ centralizes  the cyclic subgroup $<e_i>$ of $E$. 
This in turn implies that $C_L(X) > N$. But $C_L(X)$ is a $G$-invariant subgroup of $L$. 
Because $L/N$ is a chief section, we conclude that $C_L(X)= L$ and thus $X= 1$.

Assume that $R/S$ is a chief factor of $G$. Then  $C_{L/N}(R/S)$ is a $G$-invariant subgroup of 
 $L/N$. As $L/N$ is a chief factor of $G$, we conclude that $C_{L/N}(R/S)$ is trivial. 
Hence  $R/S$ is a  $q$-group, for some prime $q \ne p$.

The group $R/S$ acts faithfully  on $\Irr(L/N)$ 
as well as on $\Irr(L | \theta)$.  Furthermore, $\Irr(L/N)$ 
acts transitively on $\Irr(L | \theta)$.  Hence Glauberman's lemma implies that there
exists $\psi_1 \in \Irr(L | \theta)$  that is $R/S$-invariant. Since $S =C_G(L)$ 
fixes $\psi_1$, we conclude that $\psi_1$ is $R$-invariant. 
Hence $R$ fixes all the $G$-conjugates of $\psi_1$. Therefore $R$ and thus $R/S$,
  fixes more than one 
character in $\Irr(L | \theta)$ (because $G(\psi_1) < G$). Glauberman's lemma implies that 
$C_{\Irr(L/N)}(R/S)$ acts transitively  on the set of fixed points of $R/S$ on $\Irr(M | \theta)$.
 So  $C_{\Irr(L/N)}(R/S)$ is not trivial. This in turn implies (because $L/ N$ is abelian)  that 
$C_{L/N}(R/S)$ is not a trivial group. 
This final contradiction shows that $ L $ is a cyclic group. So 
$N \cdot \Omega_p(P)= N$, and thus every abelian characteristic subgroup of $F$ is cyclic.

So $F= N \cdot P = N_{p'} \times P $, where $P= U Z$ and  $U $ is either an extra special  
$p$-group of exponent $p$ or cyclic of order $p$  and   $Z= C_P(U)= Z(P)$ 
is a cyclic group.  

Assume first that $Z=N_p$. 
If $F > N$, that is $|U|>p $ then the fact that 
  $\theta \in \Irr(N)$ is faithful, implies that 
exists a unique irreducible character of $F$ lying above $\theta$. 
In this case we are done by induction. 
If  $F = N $ then because $N \leq Z(G)$ and $G$ is solvable we get that $G \leq C_G(F) 
\leq F$. Hence $G= F = N$ and the theorem follows.   

Assume now that $Z > N_p$. 
Then the fact that $Z=Z(P)$ is a cyclic group implies that we can get
 a normal subgroup $M$ of $G$ with $F \geq M > N$ and $|M / N| = p$.
Observe that Glauberman's lemma implies that there exists an 
extension $\phi$ of $\theta$ to $M$ whose stabilizer $G(\phi)$ has index $|G:G(\phi)|$ 
 a power of $p$. Thus  $\phi$ is $G$-invariant and we are back to Case 1.

This completes the proof of the theorem. 

\end{proof}

As it was remarked  to us by E. C. Dade, Lemma \ref{orb} can be replaced by the 
following more general remark.
\begin{lemma}
If a group $G$ of odd order acts on a non-trivial vector space $V$ over a
finite field $F$ of odd characteristic, then there are always two
$G$-orbits of equal length in $V^* := V - \{ 0 \}$. 
\end{lemma}
The above lemma gives a shorter proof of our Theorem \ref{thm2}. (For its proof,  note that the orbit of any nonzero element $v$
in $V^*$ has the same size  as the orbit of $-v \ne v$. So the $G$-orbits on $V^*$ appear in pairs of the same length.) 
The reason we kept Lemma \ref{orb}  in this paper,  is that some type of  generalization  seems  possible. 
This in turn could give 
a proof of  Higgs conjecture in the even case.
Actually, Berkovich, Isaacs and Kazarin have already shown (see  Theorem 3.4 in \cite{bik} ) 
that if a $2$-group $P$ acts irreducibly on a nontrivial vector space $V$ so that
 the orbits of this action have distinct sizes,  then  either $P$ acts transitively on $V^{*}$, 
or there are two orbits of $P$ on $V^{*}$ of sizes $2^{2k}$ and $2^{2l+1}$,with $k, l > 0$,    or 
$|V|=81$ and there are exactly two $P$-orbits on  $V^{*}$  of known sizes.

The following is Lemma 6.1 in \cite {is1} and was brought to our attention by M. Isaacs.
\begin{lemma} 
Let  $N$ be a normal subgroup of the $p$-solvable group  $G$, and   and   
$\theta$ be  a $G$-invariant  irreducible character of $N$.
If $M/N$ is a $p$-chief section of  $G$ and $C= C_G(M/N)$, then either  \\
(1) some member $\phi \in \Irr(M|\theta)$ is $G$-invariant, or \\
(2) $C$ is transitive on the set $\Irr(M |\theta)$.
\end{lemma}
This provides another way to show that we only need to handle Case 1 in Theorem \ref{thm2}.

\end{document}